\documentclass[a4paper,12pt]{article}
\usepackage[applemac]{inputenc}
\usepackage[cm]{aeguill}
\usepackage{amsfonts}
\usepackage{latexsym}
\usepackage[dvips]{graphicx}

\usepackage{amsmath}
\usepackage{amssymb}
\usepackage{amsthm}
\usepackage{bbm}
\usepackage{nicefrac}
\usepackage[T1]{fontenc}
 
\newtheorem{df}{Definition}[section]
\newtheorem{lm}[df]{Lemma}
\newtheorem{pr}[df]{Proposition}
\newtheorem{Th}[df]{Theorem}
\newtheorem{co}[df]{Corollary}
\newtheorem{rem}[df]{Remark}

\newcommand{\II}{\widehat{\mathcal{I}}}
\newcommand{\R}[1]{\mathbb{R}^{#1}}

\newcommand{\ttt}{T_c^-}
\newcommand{\tT}{T_{\tilde{c}}^-}
\newcommand{\tTT}{T_{\cc_{\tilde{\om}}}^-}
\newcommand{\tttt}{\widehat{T}_c^-}

\newcommand{\TTT}{T_c^+}

\newcommand{\<}{\prec}

\newcommand{\MM}{\widetilde{\mathcal{M}}}
\newcommand{\I}{\mathcal{I}}
\newcommand{\A}{\mathcal{A}}
\renewcommand{\AA}{\widetilde{\mathcal{A}}}
\newcommand{\HH}{\mathcal{H}}
\newcommand{\HHH}{\widehat{\mathcal{H}}}
\newcommand{\hh}[1]{\mathcal{H}(#1)\cap C^0(M,\R{})}
\newcommand{\om}{\omega}
\newcommand{\MZ}{M^\mathbb{Z}}

\renewcommand{\MM}{\widetilde{M}}
\newcommand{\cc}{\tilde{c}}
\newcommand{\xx}{\tilde{x}}
\newcommand{\yy}{\tilde{y}}
\newcommand{\KK}{\widetilde{K}}
\newcommand{\dd}{\tilde{\d}}

\renewcommand{\d}{\operatorname{d}}

\newcommand{\Hom}{\operatorname{Hom}}
\newcommand{\Homt}{\operatorname{Hom_{tame}}}

\newcommand{\conv}{\operatorname{conv}}

\author{M. ZAVIDOVIQUE}
\title{Existence of $C^{1,1}$ critical subsolutions in discrete  weak KAM theory}

\begin{document}

\maketitle

\begin{abstract}
In this article, following \cite{Za}, we study critical subsolutions in discrete weak KAM theory. In particular, we establish that if the cost function $c:M \times M\to \R{}$ defined on a smooth connected manifold is locally semi-concave and verifies twist conditions, then there exists a $C^{1,1}$ critical subsolution strict on a maximal set (namely, outside of the Aubry set). We also  explain how this applies to costs coming from Tonelli Lagrangians. Finally, following ideas introduced in \cite{FaMa} and \cite{Ma2}, we study invariant cost functions and apply this study to certain covering spaces, introducing a discrete analogue of Mather's $\alpha$ function on the cohomology.
\end{abstract}

\section*{Introduction}
In the past twenty years, new techniques have been developed in order to study time-periodic or autonomous Lagrangian dynamical systems. Among them, Aubry-Mather theory (for an introduction see \cite{Ba} for the annulus case and \cite{Mat}, \cite{Ma} for the compact, time periodic case) and Albert Fathi's weak KAM theory (see \cite{Fa} for the compact case and \cite{FaMa} for the non-compact case) have appeared to be very fruitful. More recently, a discretization of weak KAM theory applied to optimal transportation has allowed to obtain deep results of existence of optimal transport maps (see for example \cite{Be},\cite{BeBu}, \cite{BeBu1},\cite{fatfig07}). A quite similar formalism was also used in the study of time periodic Lagrangians, for example in (\cite{cis} or \cite{mas}).\\
In \cite{Za}, our goal was to study critical subsolutions and their discontinuities in a broad setting. Here, following \cite{FaSi}, and \cite{Be1} we will study the existence of more regular strict subsolutions. More precisely, we start with  a connected $C^\infty$ complete Riemmanian manifold $M$ endowed with the distance $\d(.,.)$ coming from the Riemmanian metric.
Let $c:M \times M\rightarrow \R{}$ be a locally semi-concave cost function (in other terms, in small enough charts, $c$ is the sum of a smooth and a concave function) which verifies: 
\begin{enumerate}
 \item\label{unif} \textbf{Uniform super-linearity}: for every $k\geqslant 0$,
   there exists $C(k)\in \R{}$ such that 
   $$\forall (x,y)\in M \times M,\ 
   c(x,y)\geqslant k\d(x,y)-C(k);$$
\item \label{unifb} \textbf{Uniform boundedness}: for every $R\in \R{}$, there
  exists $A(R)\in \R{}$ such that 
 $$\forall (x,y)\in M\times M, \ \d(x,y)\leqslant R \Rightarrow
  c(x,y)\leqslant A(R).$$
\end{enumerate}
A function $u$ is an $\alpha$-subsolution for $c$ if
\begin{equation}\label{sub}
\forall (x,y)\in M \times M,\ u(y)-u(x)\leqslant c(x,y)+\alpha.
\end{equation}
The critical constant $\alpha[0]$ is the smallest constant $\alpha$ such that there exist $\alpha$-subsolutions (see \cite{Za}). We will moreover suppose that $c$ verifies left and right twist conditions (defined in section \ref{2}).
\\
Under these hypothesis, we prove the following theorem:
\begin{Th}\label{Strict}
 There is a $C^{1,1}$ function $u_1:M\rightarrow \R{}$ which is an $\alpha[0]$-subsolution such that for every $(x,y)\in M \times M$, if there exists an $\alpha[0]$-subsolution, $u$ such that
$$u(y)-u(x)<c(x,y)+\alpha[0],$$
then we also  have
$$u_1(y)-u_1(x)<c(x,y)+\alpha[0].$$
\end{Th}
The proof is done, as in \cite{Be1}, using some kind of Lasry-Lions regularization combined with a version of Ilmanen's insertion lemma (proved in \cite{Beil,fz}). Let us mention that the same example as the one given in \cite{Be1} shows that in general, this is the best regularity one can expect.

This paper is organized as follows:
\begin{itemize}
\item the first two sections, \ref{known} and \ref{2}, are devoted to recalling some results proved in \cite{Za} and to introducing the notion of twist condition,
\item in the third section, \ref{example}, we study the particular case of cost coming from Tonelli Lagrangians and we prove that they fit into our framework,
\item in section \ref{existence} we prove the main theorem (\ref{Strict}),
\item finally in section \ref{section5} we study, following ideas of \cite{FaMa} the case of invariant cost functions and we apply this study in section \ref{section6} to symmetries coming from deck transformations of a cover. Finally, following ideas of Mather (\cite{Ma2}), we introduce Mather's $\alpha$ function on the cohomology.
\end{itemize}

\section*{Acknowledgment}
I first would like to thank Pierre Cardaliaguet for pointing out to me that the proof of \ref{strict} could be done using Ilmanen's lemma.
I would like to thank Albert Fathi for his careful reading of the manuscript and for his comments and remarks during my research on this subject. I am particularly indebted to him regarding to sections \ref{section5} and \ref{section6} which were written after very inspiring conversations.  This paper was partially elaborated during a stay at the Sapienza University in Rome. I wish to thank Antonio Siconolfi, Andrea Davini   and the Dipartimento di Matematica "Guido Castelnuovo" for its hospitality while I was there. I also would like to thank Explora'doc which partially supported me during this stay. Finally, I would like to thank the ANR KAM faible (Project BLANC07-3\_187245, Hamilton-Jacobi and Weak KAM Theory) for its support during my research.
\\

First, let us recall the setting and some results proved in \cite{Za}.

\section{Known results}\label{known}
In this section we quickly survey some previously obtained results, see \cite{Za}.
Throughout this paper, we will assume $M$ is a connected $C^\infty$ complete Riemmanian manifold endowed with the distance $\d(.,.)$ coming from the Riemmanian metric. We will consider a cost function $c:M \times M\to \R{}$ verifying the properties \ref{unif} and \ref{unifb} mentioned in the introduction.
We will denote $\alpha[0]$, the Ma\~n\' e critical value as defined for example in \cite{Za}.
We say
that a function $u:M\to\R{}$ is critically dominated or that it is a critical
subsolution if it is $\alpha[0]$-dominated that is if
$$\forall (x,y)\in M \times M,\ u(y)-u(x)\leqslant c(x,y)+\alpha[0].$$
Let us mention that $\alpha[0]$ is defined as being the smallest value such that there are subsolutions. More precisely, if $C\in \R{}$, we let $\HH(C)\subset M^{\R{}}$ be the set of $C$-dominated functions, that is the set of $u$ verifying 
$$\forall (x,y)\in M \times M,\ u(y)-u(x)\leqslant c(x,y)+C.$$
Then the Ma\~n\' e critical value is
$$\inf\{C\in \mathbb{R},\  \HH(C)\neq \varnothing \}.$$

As is customary, we introduce the discrete Lax-Oleinik semi-groups:
$$\ttt u(x)=\inf_{y\in M}u(y)+c(y,x),$$
$$\TTT u(x)=\sup_{y\in M}u(y)-c(x,y).$$ 
Finally, we call negative
(resp. positive) weak KAM solution a fixed point of the operator $\ttt
+\alpha[0]$ (resp. $\TTT -\alpha[0]$).
\begin{pr}
A function $u:M\mapsto \R{}$ is a critical subsolution (written $u\<c+\alpha [0]$)  if and only if it verifies one of the following equivalent properties:
\begin{enumerate}
 \item[(i)]$\forall (x,y)\in M \times M,\ u(x)-u(y)\leqslant c(y,x)+\alpha [0]$;
\item[(ii)] $u\leqslant \ttt u+\alpha [0]$;
\item[(iii)] $u\geqslant \TTT u-\alpha [0]$.
\end{enumerate}

\end{pr}
\begin{df}
 \rm Consider $u:M\to \R{}$ a critical subsolution ($u\<c+\alpha [0]$). 
We will say that $u$ is strict at $(x,y)\in M \times M$ if and only if 
$$u(x)-u(y)< c(y,x)+\alpha [0].$$
We will say that $u$ is strict at $x\in M$ if
\begin{equation*}
\forall y\in M, \ u(y)-u(x)<c(x,y)+\alpha [0] \; \ 
\mathrm{ and }\; \ 
u(x)-u(y)<c(y,x)+\alpha [0] . 
\end{equation*}
\end{df}
We recall a characterization of strict continuous subsolutions (see \cite{Za}).
 \begin{pr}\label{strict}
The \textbf{continuous} critical subsolution $u$ is strict at $x$ if and only if
$u(x)<\ttt u(x)+\alpha [0]$ and $u(x)>\TTT u(x)-\alpha [0]$.
 \end{pr}
\begin{df}\label{d:aub2}\rm
 Let $u$ from $M$ to $\R{}$ verify $u\<c+\alpha [0]$. We will say that
 a chain $(x_i)_{0\leqslant i\leqslant n}$ is $(u,c,\alpha
 [0])$-calibrated if 
$$u(x_n)=
 u(x_0)+c(x_0,x_1)+\cdots+c(x_{n-1},x_n)+n\alpha [0].$$
 Notice that a
 sub-chain of a calibrated chain formed by consecutive elements is again calibrated.  

Following Bernard and Buffoni \cite{Be} we will call
 Aubry set of $u$, $\AA_u$ the subset of $\MZ$ consisting of the
 sequences whose finite sub-chains are $(u,c,\alpha
 [0])$-calibrated. 
We set 
$$\widehat{\A}_u =\{(x,y)\in M \times M ,\ \exists (x_n)_{n\in \mathbb{Z}} \in \AA_u\;\textrm{ with }\;x_0=x\; \mathrm{
  and }\; x_1=y\},$$
 and we define the projected Aubry set of $u$ by 
$$\A _u=\{x\in M,\ \exists (x_n)_{n\in \mathbb{Z}} ,\ (u,c,\alpha
[0])\textrm{-calibrated with }x_0=x\}. $$ 
We then define the Aubry set:
$$\AA=\bigcap_{u\<c+\alpha [0]}\AA _{u},$$
the projected Aubry
sets
$$\widehat{\A} =\{(x,y)\in M \times M ,\exists (x_n)_{n\in \mathbb{Z}}\in \AA ,x=x_0\; \mathrm{
  and }\; y=x_1\},$$ 
and
 $$\A=\bigcap_{u\<c+\alpha [0]}\A _{u}$$ 
where in all cases, the
intersection is taken over all critically dominated functions. 
\end{df}
We recall some further facts obtained in \cite{Za}:
\begin{lm}\label{trivial}
 Let $u\< c+\alpha[0]$ be a dominated function and $(x,y)\in M \times M$. If the following identity is verified: 
$$u(x)-u(y)=c(y,x)+\alpha[0]$$
then $u(x)=\ttt u(x)+\alpha[0]$.
\\
If the following identity is verified: 
$$\ttt u(x)-\ttt u(y)=c(y,x)+\alpha[0]$$ 
then $u(y)=\ttt u(y)+\alpha[0]$ and $\ttt u(x)=u(y)+c(y,x)$.
\end{lm}
\begin{pr}\label{egalite aubry}
 Let $u\<c+\alpha [0]$ be a dominated function, then $\AA_u=\AA_{\ttt u}$.
\\
In particular, the following equalities hold:
$\widehat{\A}_u=\widehat{\A}_{\ttt u}$ and $\A_u=\A_{\ttt u}$.
\end{pr}
% Here is a lemma that will be useful in the sequel:
% \begin{lm}\label{I}
%  There is a continuous function $u\<c+\alpha [0]$ such that
%  $\AA_u=\AA$.
% \end{lm}
% \begin{lm}\label{IIII}
%  If $u\<c+\alpha [0]$ and $x\in M$ then 
% $$x\in \A_u \iff \forall p\in
%  \mathbb{N}, \T{p}u(x)+p\alpha[0]=u(x)=\TT{p}u(x)-p\alpha[0]$$
% \end{lm}
% \begin{df} \rm
% 
% $$\widehat{\A}_u =\{(x,y)\in M \times M ,\exists z\in \AA_u ,y=p(z)\; \mathrm{
%   and }\; x=p\circ S(z)\}$$ 
% and more generally,
%  $$\widehat{\A} =\{(x,y)\in M \times M ,\exists z\in \AA ,y=p(z)\; \mathrm{
%   and }\; x=p\circ S(z)\}.$$ 
% \end{df}

\begin{Th}\label{aubrymax}
Let $u\<c+\alpha[0]$ be a critically dominated function.
 There is a continuous subsolution $u'$ which is strict at every $(x,y)\in M \times M-\widehat{\A}_u$ and which is equal to $u$ on $\mathcal{A}_u$. In particular, we have
$$\widehat{\A}_u=\widehat{\A}_{u'}.$$

 There is a continuous subsolution $u_0$ which is strict at every $(x,y)\in M \times M-\widehat{\A}$. In particular
$$\widehat{\A}=\widehat{\A}_{u_0}.$$
\end{Th}

\begin{pr}\label{newpr}
Let $u:M\to \R{}$ be a critical subsolution. If $u$ is strict at every $(x,y)\in M \times M-\widehat{\A}_u$ then $u$ is strict at every $x\in M-\A_u$. In particular, if $u$ is continuous, the following inequalities hold:
$$\forall x\in  M-\A_u, \  u(x)< \ttt u+\alpha[0],$$
$$\forall x\in  M-\A_u, \  u(x)> \TTT u-\alpha[0].$$
\end{pr}

\section{More regularity, the twist conditions and the partial dynamic}\label{2}
 We will
now suppose that the cost function is locally semi-concave, see \cite{fatfig07} or \cite{can} for a definition. In this text we will use the term locally semi-concave to refer to what is usually called locally semi-concave with linear modulus. Let
us begin with some basic properties of locally semi-concave functions
that we will need later.
\begin{pr}[differentiability properties]
The following properties hold 

 \begin{enumerate}
 \item[(i)] Let $f$ be a locally semi-concave function from $M$ to $\R{}$ and let $x_0$ be a local minimum of $f$, then $f$ is differentiable at $x_0$ and $\d_{x_0}f=0$.
\item[(ii)]Let $f$ and $g$ be two locally semi-concave functions from $M$ to $\R{}$ and $x_0$ be a point where $f+g$ is differentiable, then both $f$ and $g$ are differentiable at $x_0$.
\end{enumerate} 
\end{pr}
\begin{proof}
(i) Since the result is local, we can suppose $f$ is defined on an open subset $U\subset \R{n}$, that
  it is semi-concave, and that $x_0=0$ is a global minimum. Moreover,
  since the problem is invariant by addition of a constant to $f$, we
  will assume $f(0)=0$. Let $K\in \R{}$ such that $x\rightarrow
  f(x)-K\|x\|^2$ is concave on $U$. By the Hahn-Banach theorem, there
  is a linear form $p$ such that
\begin{equation}\label{11}
 \forall x\in U, 0\leqslant f(x)\leqslant p(x)+K\|x\|^2.
\end{equation}
The positive function $p(x)+K\|x\|^2$ admits a local minimum at $0$. Its differential at $0$ must vanish so $p=0$ and
\begin{equation}
 \forall x\in U, 0\leqslant f(x)\leqslant K\|x\|^2 
\end{equation}
 therefore $f$ is differentiable at $0$ with $\d_0f=0$.\\ (ii) Once
 more, let us assume that $f$ and $g$ are defined on an open subset
 $U\subset \R{n}$, that they are semi-concave and that $x_0=0$. It is
 clear that if $p$ and $q$ are linear forms respectively in the
 super-differential at $0$ of $f$ and $g$ then $p+q$ is in the
 super-differential at $0$ of $f+g$. Since $f+g$ is differentiable at
 $0$, its super-differential at $0$ is a singleton. Moreover, $f$ and
 $g$'s super-differentials at $0$ are non empty by the Hahn-Banach
 theorem and must also be singletons. This proves that $f$ and $g$ are
 differentiable at $0$.
\end{proof}
\begin{df}[minimizing chains]\rm
Let $(x,y)\in M \times M$ and $k\in \mathbb{N^*}$, we will say that $(x_1,\ldots ,x_k)\in M^k$ is a minimizing chain between $x$ and $y$ if, setting $x_0=x$ and $x_{k+1}=y$,
$$\forall (y_1,\ldots ,y_{k})\in M^k,\ \sum_{i=0}^k c(x_i,x_{i+1})\leqslant c(x,y_1)+\sum_{i=1}^{k-1} c(y_i,y_{i+1})+c(y_k,y).$$
Notice that any sub-chain of a minimizing chain formed by consecutive elements is again minimizing.\\
We will say that a sequence $(x_n)_{n\in \mathbb{Z}}$ is a minimizing sequence if all sub-chains formed
by consecutive elements are minimizing.
\end{df}
A straightforward consequence of the previous results is the following theorem.
\begin{Th}
 If $(x,x_1,y)\in M \times M \times M$ is a minimizing chain then $\partial c/\partial y(x,x_1)$ and $\partial c/\partial x(x_1,y)$ exist and verify 
$$\label{e:EL}
 \frac{\partial c}{\partial y}(x,x_1)+\frac{\partial c}{\partial x}(x_1,y)=0.
\eqno{\mathrm{(EL)}}$$
\end{Th}
The equation above may be considered as a discrete analog of the
Euler-Lagrange equation. It was already introduced in works on twist maps such as \cite{MatIHES}. By analogy, we therefore can define extremal
chains and extremal sequences.
\begin{df}[extremal chains]\rm
 We will say that $(x,x_1,\ldots ,x_{k-1},y)$ is an extremal chain if
 for every $i\in [1,k-1]$, 
$(x_{i-1},x_i,x_{i+1})$ verify (EL) that is
$$\frac{\partial c}{\partial y}(x_{i-1},x_i)+\frac{\partial c}{\partial x}(x_i,x_{i+1})=0,$$
where $x_0=x$ et $x_{k}=y$.
\\
We will say that a sequence $(x_n)_{n\in \mathbb{Z}}$ is extremal if for every $i\in \mathbb{Z}$, $(x_{i-1},x_i,x_{i+1})$ verify (EL), that is
$$\frac{\partial c}{\partial y}(x_{i-1},x_i)+\frac{\partial c}{\partial x}(x_i,x_{i+1})=0.$$
\end{df}
\begin{rem}\rm
Notice that minimizing chains and sequences are extremal.
\end{rem} 
It seems now natural to try and define a dynamic on $M$ as follows: given
two points $x_1$ and $x_2$, we would like to find an $x_3$ such that
the triplet $(x_1,x_2,x_3)$ verifies the discrete Euler-Lagrange equation (EL).
However such an $x_3$ if it exists is not necessarily unique.
To solve this problem, we introduce an additional constraint. It has already been introduced in the optimal transportation setting (see \cite{fatfig07}) and it is reminiscent of twist maps of the circle (see \cite{Ma} or \cite{Ba}):
\begin{df}\rm
 We will say that $c$ verifies the \textit{right twist condition} if
 for every $x\in M$, the function $y\mapsto \partial c/\partial
   x(x,y)$ is injective where it is defined.

 Similarly, we will say
 $c$ verifies the \textit{left twist condition} if for every $y\in M$,
 the function $x\mapsto \partial c/\partial y(x,y)$ is
 injective where it is defined.\\ Finally we say $c$ verifies the
 \textit{twist condition} if $c$ verifies the left and right twist
 conditions.
\end{df}

For more explanations about this definition see \cite{fatfig07}. Let us just state that costs coming from time-periodic Tonelli Lagrangians satisfy the twist condition as is explained in the next section.

It is possible under the right twist condition to define a partial
dynamic on $M \times M$ in the future and to define one in the past using the
left twist condition. Let us be more precise on those points.
Following \cite{fatfig07}, let us define the \textit{skew Legendre transforms}:
\begin{df}\rm
We define the left skew Legendre transform as the partial map
$$\Lambda_c^l:M \times M\to T^*M,$$
$$(x,y)\mapsto \left( x,-\frac{\partial c}{\partial
   x}(x,y)\right),$$
   whose domain of definition is
   $$\mathcal{D}( \Lambda_c^l)=\left\{(x,y)\in M \times M,\ \frac{\partial c}{\partial
   x}(x,y) \,\,\mathrm{exists}\right\}.$$
Similarly, let us define the right skew Legendre transform as the partial map
$$\Lambda_c^r:M \times M\to T^*M,$$
$$(x,y)\mapsto \left( y,\frac{\partial c}{\partial
   y}(x,y)\right),$$
   whose domain of definition is
   $$\mathcal{D}( \Lambda_c^r)=\left\{(x,y)\in M \times M,\ \frac{\partial c}{\partial
   y}(x,y) \,\,\mathrm{exists}\right\}.$$
\end{df}
Note that saying that $c$ verifies the left (resp. right) twist condition amounts to saying that the left (resp. right) skew Legendre transform is injective. Now we define the partial dynamics on $M \times M$.
\begin{df}[partial dynamics]\rm
Let $c:M \times M\to \R{}$ be a locally semi-concave cost function which verifies the left twist condition. Set $\varphi_{-1}:M \times M\to M \times M$  the partial map defined by
$$\varphi_{-1}(x,y)=(\Lambda^l_c)^{-1}\circ \Lambda^r_c(x,y),$$
Similarly, if $c:M \times M\to \R{}$ is a locally semi-concave cost function which verifies the right twist condition, set $\varphi_{+1}:M \times M\to M \times M$  the partial map defined by
$$\varphi_{+1}(x,y)=(\Lambda^r_c)^{-1}\circ \Lambda^l_c(x,y).$$
\end{df}
\begin{rem}\rm
If both left an right twist conditions are verified, it is clear that $\varphi_{-1}$ and $\varphi_{+1}$ are  inverses of one another on the intersection of their domain of definition.
\end{rem}

\section{Example: costs coming from Tonelli Lagrangian}\label{example}
This section is devoted to explaining how these notions apply to costs coming from Tonelli Lagrangians. A convenient reference for the proofs of these results is the appendix of \cite{fatfig07}. \\
Let $L:TM\times \R{}\to \R{}$ be a time periodic Tonelli Lagrangian, that is a $C^2$ function verifying
\begin{enumerate}
\item\textbf{uniform super-linearity}: for every $K>0$, there exists $C^*(K)\in \R{}$ such that 
$$\forall (x,v,t)\in TM\times \R{},\  L(x,v,t)\geqslant K\|v\|-C^*(K),$$
\item\textbf{uniform boundedness}: for every $R\geqslant 0$, we have
$$A^*(R)=\sup \{L(x,v,t),\|v\|\leqslant R\}<+\infty,$$
\item\textbf{$C^2$-strict convexity in the fibers}: for every $(x,v,t)\in TM\times \R{}$, the second derivative along the fibers $\partial ^2L/\partial v^2 (x,v,t)$ is positive strictly definite,
\item\textbf{time periodicity}: for every $(x,v,t)\in TM\times \R{}$, we have the relation $L(x,v,t)=L(x,v,t+1)$,
\item\textbf{completeness}: the Euler-Lagrange flow associated to $L$ is complete.
\end{enumerate} 

 Then we can define a cost function $c_L$ by
$$\forall (x,y)\in M \times M,\ c_L(x,y)=\inf_{\substack{\gamma(0)=x\\\gamma(1)=y}}\int_{s=0}^1L(\gamma(s),\dot{\gamma}(s),s)\d s,$$
where the infimum is taken over all absolutely continuous curves.\\
\begin{pr}
The cost  $c_L$ verifies conditions \ref{unif} and \ref{unifb} and  is locally semi-concave.
\end{pr}
 Let $(x,y)\in M \times M$ and let $\gamma_{x,y}$ verify that
$$c_L(x,y)=\int_{s=0}^1L(\gamma_{x,y}(s),\dot{\gamma}_{x,y}(s),s)\d s,$$
with $\gamma_{x,y}(0)=x$ and $\gamma_{x,y}(1)=y$ then the following holds:
\begin{pr}
 The linear form on $TM\times TM$ defined by 
$$(v,w)\mapsto \frac{\partial L}{\partial v} (y,\dot{\gamma}_{x,y}(1),0)w-\frac{\partial L}{\partial v} (x,\dot{\gamma}_{x,y}(0),0)v$$
is a super-differential of $c_L$ at $(x,y)$. In particular, if $\partial c_L/\partial x(x,y)$ exists then it must be equal to $-\partial L/\partial v (x,\dot{\gamma}_{x,y}(0),0)$ and similarly, if $\partial c_L/\partial y(x,y)$ exists then it must be equal to $\partial L/\partial v (y,\dot{\gamma}_{x,y}(1),0)$. 
\end{pr}
Therefore, if either of the partial derivatives exists, the curve $\gamma_{x,y}$ realizing the minimum is unique (since $L$ is strictly convex, the mapping $\partial L/\partial v$ is injective in each fiber and since $\gamma_{x,y}$ is an action minimizing curve for $L$ and the flow is complete, it is a trajectory of the Euler-Lagrange flow). As a corollary, we have:
\begin{Th}
 The cost $c_L$ verifies both left and right twist conditions.
 \end{Th}

We may now compute the skew Legendre transforms (when they exist).
From the previous results we have the following:
\begin{multline*}
\forall (x,y)\in\mathcal{D}( \Lambda_c^l),\  \Lambda_c^l(x,y)= \left( x,-\frac{\partial c}{\partial
   x}(x,y)\right)\\
   = \left( x,\frac{\partial L}{\partial v} (x,\dot{\gamma}_{x,y}(0),0)\right)=\mathcal{L}_L(x,\dot{\gamma}_{x,y}(0),0),
   \end{multline*}
   \begin{multline*}
\forall (x,y)\in\mathcal{D}( \Lambda_c^r),\  \Lambda_c^r(x,y)= \left( y,\frac{\partial c}{\partial
   y}(x,y)\right)\\
   = \left( y,\frac{\partial L}{\partial v} (y,\dot{\gamma}_{x,y}(1),0)\right)=\mathcal{L}_L(y,\dot{\gamma}_{x,y}(1),0),
   \end{multline*}
   where we recall that the mapping $\mathcal{L}_L$ is the classical Legendre transform from $TM$ to $T^*M$ defined by 
   $$\forall (x,v,t)\in TM\times \R{},\ \mathcal{L}_L(x,v,t)=\left(x,\frac{\partial L}{\partial v}(x,v,t)\right).$$
\\
Finally, let us study the partial dynamics for the cost $c_L$. Let $(x,y)\in M \times M$ be such that  $\partial c_L / \partial y(x,y)$ exists, let us compute (if it exists) $\varphi_{+1}(x,y)$. We are looking for a $z$ such that
$$ \frac{\partial c_L}{\partial y}(x,y)=-\frac{\partial c_L}{\partial x}(y,z),$$
where all the partial derivatives exist, that is, using the previous notations,
$$\frac{\partial L}{\partial v} (\gamma_{x,y}(1),\dot{\gamma}_{x,y}(1),0)= \frac{\partial L}{\partial v} (\gamma_{y,z}(0),\dot{\gamma}_{y,z}(0),0),$$
which proves, since $\partial L/\partial v$ is injective, that $\dot{\gamma}_{x,y}(1)=\dot{\gamma}_{y,z}(0)$.
Moreover, since all the above curves are minimizers, they are trajectories of the Euler-Lagrange flow of $L$ which we denote by $\varphi_L$. To put it all in a nutshell, if $z$ exists, then 
$$(z,\dot{\gamma}_{y,z}(1),1)=\varphi^1_L(y,\dot{\gamma}_{x,y}(1),1)=\varphi^2_L(x,\dot{\gamma}_{x,y}(0),0).$$
From this discussion, we obtain the following result:
\begin{pr}
The point $(y,z)=\varphi_{+1}(x,y)$ exists if and only if the trajectory $\gamma$ defined by
$$\forall s\in [0,1], \ (\gamma(s),\dot{\gamma}(s),s)=\varphi_L(s)(y,\dot{\gamma}_{x,y}(1),1)$$
is the only action minimizing curve between $y$ and $\gamma(1)=z$ (defined on a time interval of length $1$).
\end{pr}
\begin{proof}
It only remains to prove the "if" part, therefore, let us assume that $\gamma$ is the only action minimizing curve between $y=\gamma(0)$ and $z=\gamma(1)$.
\\
 We first prove that if $(y_n,z_n)_{n\in \mathbb{N}}$ is a sequence converging to $(y,z)$ and if $(\gamma_n)_{n\in \mathbb{N}}$ verifies, $\gamma_n(0)=y_n$, $\gamma_n(1)=z_n$ and 
 $$\forall n\in \mathbb{N},\  c_L(y_n,z_n)=\int_{0}^1L(\gamma_n(s),\dot{\gamma_n}(s),s)\d s,$$
 then the $(\gamma_n,\dot{\gamma}_n)$ converge uniformly to $(\gamma,\dot{\gamma})$ when $n\to +\infty$.\\
 As a matter of fact, since $M$ is compact and the $\gamma_n$ are action minimizing curves defined for length time of $1$, by the a priori compactness lemma (see \cite{Fa}), the sequence $(\gamma_n(0),\dot{\gamma}_n(0)) _{n\in \mathbb{N}}$ is bounded. We obtain, by continuity of the Euler-Lagrange flow that the sequence of functions, $(\gamma_n,\dot{\gamma}_n) _{n\in \mathbb{N}}$ is relatively compact for the compact open topology. Therefore we only have to prove that any converging subsequence converges to  $(\gamma,\dot{\gamma})$. Up to an extraction, let us assume that  $(\gamma_n,\dot{\gamma}_n)$ converges to some $(\delta,\dot{\delta})\in TM^{[0,1]}$. By continuity of the Euler-Lagrange flow, we necessarily have 
 $$\forall s\in[0,1], \ (\delta(s),\dot{\delta}(s),s)=\varphi_L(s)(\delta(0),\dot{\delta}(0),0).$$
 By continuity of the function $c_L$, we therefore obtain that 
 $$c_L(y,z)=\int_{0}^1L(\delta(s),\dot{\delta}(s),s)\d s$$
 which proves that $\delta=\gamma$ by uniqueness of $\gamma$ and therefore that the $(\gamma_n,\dot{\gamma}_n)$ converge uniformly to $(\gamma,\dot{\gamma})$.
 \\
 As a direct corollary of the previous result, we have that if $(y_n,z_n)_{n\in \mathbb{N}}$ is a sequence converging to $(y,z)$ and such that $c_L$ is differentiable at each $(y_n,z_n)$, then 
 \begin{multline*}
 \lim_{n\to +\infty}\d_{(y_n,z_n)}c_L\\
 =\lim_{n\to +\infty}  \frac{\partial L}{\partial v} (z_n,\dot{\gamma_n}_{y_n,z_n}(1),1)\d y-\frac{\partial L}{\partial v} (y_n,\dot{\gamma_n}_{y_n,z_n}(0),0)\d x\\
 = \frac{\partial L}{\partial v} (z,\dot{\gamma}_{y,z}(1),1)\d y-\frac{\partial L}{\partial v} (y,\dot{\gamma}_{y,z}(0),0)\d x.
 \end{multline*}
 Since $c_L$ is a locally semi-concave function, it follows from basic properties of the Clarke super-differential (\cite{clark})  that $c_L$ is differentiable at $(x,y)$.
\end{proof}
As an immediate corrolary we obtain the following result that has been widely known for some time but, to our knowledge, never written (\cite{Faper}):
\begin{co}
For a cost coming from a Lagrangian, let $(x,y)\in M \times M$, if either $\partial c_L/\partial x(x,y)$ or $\partial c_L /\partial y (x,y)$ exists then $c_L$ is in fact differentiable at $(x,y)$.
\end{co}

In the Lagrangian case, the partial dynamic $\varphi_{+1}$ may be recovered from the restriction of the Euler-Lagrange flow, $\varphi_L^1$ to the right subset. Of course, the same holds for the negative time dynamic $\varphi_{-1}$ which is closely related to the restriction to some set of $\varphi^{-1}_L$.

\section{Existence of $C^{1,1}_{loc}$ critical subsolutions}\label{existence}
 We will now suppose $c$ verifies the left and
right twist conditions. Our goal from now on will be to construct more
regular strict subsolutions.
\begin{pr}\label{dif}
 Let $u\<c+\alpha [0]$ be a dominated function and $(x_1,x_2,x_3)$ be a
 calibrated chain, then $u$ is differentiable at
 $x_2$. Moreover,
  $$\d_{x_2}u=\frac{\partial c}{\partial
   y}(x_1,x_2)=-\frac{\partial c}{\partial x}(x_2,x_3).$$
\end{pr}
\begin{proof}
 By definition of domination, the following inequalities hold:
$$\forall x\in M,\ u(x_1)+c(x_1,x)+\alpha [0]\geqslant u(x) \geqslant
 u(x_3)-c(x,x_3)-\alpha [0]$$ 
where both inequalities are equalities
 at $x_2$. Define the functions 
$$\varphi (x)=u(x_1)+c(x_1,x)+\alpha
 [0]\; \mathrm{ and}\; \psi (x)=u(x_3)-c(x,x_3)-\alpha [0].$$
 Clearly, $\varphi$ is
 locally semi-concave and $\psi$ is locally semi-convex,
 $\varphi\geqslant \psi$ with equality at $x_2$. The function
 $\varphi-\psi$ is always non-negative and vanishes at $x_2$ (which is
 a global minimum). Moreover, it is locally semi-concave therefore it is
 differentiable at $x_2$ and $\d_{x_2}(\varphi-\psi)=0$. Finally,
 since both $\varphi$ and $-\psi$ are locally semi-concave, both of
 them are differentiable at $x_2$ and from the inequalities $\varphi\geqslant
 u\geqslant \psi$ we deduce that $u$ is differentiable at $x_2$ with
 $\d_{x_2}u=\d_{x_2}\varphi =\d_{x_2}\psi$.
\end{proof}
As a corollary we have the following:
\begin{co}\label{newco}
Suppose $c$ satisfies the right and left twist conditions. If $u:M\to \R{}$ is a critically dominated function and $x\in \A_u$, then $\d_x u$ exists. Moreover there is a unique point $x_1$ such that $\partial c/\partial x (x,x_1)$ exists and verifies 
$$\d_{x}u=-\frac{\partial c}{\partial x}(x,x_1).$$
This point $x_1$ is also the unique point such that $(x,x_1)\in \widehat{\A}_u$. In particular it is necessarily in $\A_u$.\\
In the same way, there is a unique point $x_{-1}$ such that $\partial c/\partial y (x_{-1},x)$ exists and verifies 
$$\d_{x}u=\frac{\partial c}{\partial y}(x_{-1},x).$$
This point $x_{-1}$ is also the unique point such that $(x_{-1},x)\in \widehat{\A}_u$. In particular it is necessarily in $\A_u$.
\end{co}
\begin{Th}[Mather's graph theorem]
 Let $u\<c+\alpha [0]$ be a dominated function then $u$ is
 differentiable on $\A$. Moreover, the differential of $u$ is
 independent of the dominated function $u$. In particular, the canonical
  projections from $\AA$ to $\A$ and from $\widehat{\A}$ to $\A$ are bijective. 
% Actually, $\AA$ is a locally lipschitz graph over $\A$.
\end{Th}
\begin{proof}
 The first part is a straightforward consequence of the previous
 corollary (\ref{newco}). To prove the second part, notice that if $x\in \A$ then there
 is a sequence $(x_n)_{n\in \mathbb{Z}}\in \AA$ with
 $x_0=x$. Therefore, $\d_{x}u=\partial c/\partial y (x_{-1},x)$
 which is independent from $u$. The last part is now a straightforward
 consequence of the twist condition.
\end{proof}
\begin{rem}\rm
Originally, in \cite{Ma2}, Mather obtains in his graph theorem that the projection, from the Aubry set to the projected Aubry set, is a bi-Lipschitz homomorphism. In the previous theorem, this is not necessarily the case, due to the fact that in the general framework we propose, the Skew Legendre transforms need not be bi-Lipschitz on their domain of definition. We will however give a bi-Lipschitz version of the graph theorem at the end of this section (see \ref{graphbis}). 
\end{rem}

We now would like to obtain some   regularity results  about the differential of $u$ on $\A_u$. One way to obtain that is to look for a $u$ which is locally
semi-concave. Here is a lemma that will help us to do so.
\begin{pr}
 If $u\<c+\alpha [0]$ then $\ttt u$ is locally semi-concave.
\end{pr}
\begin{proof}
%  Let $V$ be a compact neighborhood of some $x_0\in M$. We can pick a
% compact set $V'$ such that $V\subset V'$ and such that $\forall y\in
% V,\exists y'\in V',\ttt u(y)=u(y')+c(y',y)$. By compactness, we can
% cover the compact set $V'$ by a finite number of open sets
% $U_1,\ldots ,U_k$ such that for each $1\leqslant i\leqslant k$,
% there is a homomorphism $\varphi_i :U_i\rightarrow \R{n}$, shrinking
% $V$ if necessary, we can also assume that it is in the domain of a
% chart $\varphi$ and that there is a constant $K\in \R{}$ such that
% for each $1\leqslant i\leqslant k$, the function $(x,y)\rightarrow
% c( \varphi_i^{-1}(x),\varphi (y))$ is $K$-semi-concave. If $x,y\in M$
% then we already know that $\ttt u(y)-\ttt u(x)\leqslant
% c(x',y)-c(x',x)$ where $x'\in V'$ verifies $\ttt
% u(x)=u(x')+c(x',x)$. There is an $i\leqslant k$ such that $x'\in
% U_i$. Therefore, $$\ttt u(y)-\ttt u(x)\leqslant c(x',y)-c(
% x',x)\leqslant
% p_x(\varphi^{-1}(y)-\varphi^{-1}(x))+K\|\varphi^{-1}(y)-\varphi^{-1}(x)\|^2$$
% Where $p_x$ is a linear form depending only on $x$. This is enough
% to prove the result stated above.
The proof actually goes along the same lines as the proof that
the image of a dominated function is continuous. The function $\ttt u$
is locally a finite infimum of equi-locally semi-concave functions and
is therefore itself locally semi-concave. For more details, see \cite{fatfig07} or \cite{Za}.
\end{proof}
The next proposition shows that in order to achieve our goal,  we can consider $\ttt u$ instead of $u$. Let us recall that by \ref{egalite aubry} we have $\A_u=\A_{\ttt
   u}$ as soon as $u$ is dominated. Here is a complement when $c$ is locally semi-concave.
\begin{lm}\label{err}
 Let $u\<c+\alpha [0]$ be a dominated function, then  if $x\notin \A_u$ and $\tilde{x} \in M$ verifies $\ttt
 u(x)=u(\tilde{x})+c(\tilde{x},x)$ then $\tilde{x}\notin \A_u=\A_{\ttt u}$. \\
 If $\tilde{x} \in M$ verifies $\TTT
 u(x)=u(\tilde{x})-c(x,\tilde{x})$ then $\tilde{x}\notin \A_u=\A_{\TTT u}$.
\end{lm}
\begin{proof}
% By contradiction, assume $y\in \A_u$ that is suppose there is a sequence $(x_n)_{n\in
%   \mathbb{Z}}\in \AA_u$ with $x_0=y$. By \ref{egalite aubry}, this
% sequence is also ($\ttt u,c,\alpha [0]$)-calibrated. Therefore if we
% consider the chain $x'_n=x_n$ for $n\leqslant 0$ and $x'_1=x$ then
% this chain is again ($\ttt u,c,\alpha [0]$)-calibrated. By the right
% twist condition, we must have $x=x'_1=x_1$ because
%$$\d_y \ttt u=\frac{\partial c}{\partial x}(y,x_1)=\frac{\partial c}{\partial x}(y,x'_1).$$
% This contradicts the fact
% that $x\notin \A_u$ since $x_1\in \A_u$.
%\\
%The proof of the second part is similar.
Assume by contradiction $\tilde{x}\in \A_u$. By definition of the Lax-Oleinik semi-group, from 
$$\forall z\in M,\  \ttt u(x)\leqslant u(z)+c(z,x),$$
we obtain that
$$\forall z\in M,\  \ttt u(x) -u(z)\leqslant c(z,x).$$
At $z=\tilde{x}\in \A_u$ the differential $\d_{\tilde{x}} u$ exists, therefore the sub-differential  of the locally semi-concave function $z\mapsto c(z,x)$ is not empty at $\tilde{x}$. This implies that the partial derivative $\partial c/\partial x (\tilde{x},x)$ exists and verifies 
$$\d_{ \tilde{x}}u=-\frac{\partial c}{\partial x} (\tilde{x},x).$$
By corollary \ref{newco}, we have necessarily $x\in \A_u$, a contradiction.\\
The proof of the second part is similar.
\end{proof}
\begin{pr}
 If $u\<c+\alpha [0]$ is a continuous subsolution which is strict outside of $\widehat{\A}_u$ then $\ttt u$ and $\TTT u$  are also subsolutions strict outside of $\widehat{\A}_{\ttt u}=\widehat{\A}_u=\widehat{\A}_{\TTT u}$.
\end{pr}
\begin{proof}
 We already know that $\ttt u$ is a subsolution.
Let $(x,x')\in M \times M$ verify $\ttt u(x)-\ttt
 u(x')=c(x',x)+\alpha [0]$. We therefore must have 
$$\ttt u(x')+\alpha
 [0]=u(x')$$
  as seen in \ref{trivial}.
% Let $y$ and $y'$ verify that 
%$$\ttt u (x)=u(y)+c(y,x)\;\mathrm{ and}\; \ttt u
% (x')=u(y')+c(y',x').$$
% Existence of such points is proved in \cite{Za}.
% Assume first $x\notin \A_u$.
%  We have proved
% that $y\notin \A_u$.  But this
% shows that on the one hand, 
%$$u(x')-u(y')=c(y',x')+\alpha [0]$$
% thus
% that $x'\in \A_u$ since $u$ is strict outside of $\A_u$. On the second hand, 
%$$\ttt u(x)=u(x')+c(x',x)$$
% implies by the previous proposition that $x'\notin \A_u$ which is a
% contradiction.

% Assume now  that $x\in \A_u$.  We therefore must have 
%$$\ttt u(x)+\alpha
% [0]=u(x)$$ as seen in \ref{trivial}. But this
% shows that $u(x)-u(x')=c(x',x)+\alpha [0]$ thus that $(x',x)\in \widehat{\A}_u$ (since $u$ is strict outside of this set).
Since $u$ is  continuous and strict outside of $\widehat{\A}_u$, by proposition \ref{newpr} we necessarily have $x'\in \A_u$. Using now that $u(x')=\ttt u(x')+\alpha[0]$, we obtain the fact that
$$\ttt u(x)=u(x')+c(x',x).$$
By \ref{err} we must have $x\in \A_u$ and therefore $\ttt u(x)=u(x)+\alpha[0]$.
To put it all in a nutshell, we obtained that 
$$u(x)-u(x')=c(x',x)+\alpha[0].$$
Since $u$ is strict outside of $\widehat{\A}_u= \widehat{\A}_{\ttt u}$ we finally get that $(x,x')\in \widehat{\A}_{\ttt u}$. 
\\
The proof for $\TTT u$ is the same.
\end{proof}
Using the previous result with \ref{aubrymax} we obtain the following:
\begin{lm}\label{210}
 Given a continuous critical subsolution $u$, there is a locally semi-concave critical subsolution $u'$ which is strict outside of $\widehat{\A}_u$ and equal to $u$ on $\A_u$.
Moreover, there is a locally semi-concave subsolution $u_0$ which is strict outside of $\widehat{\A}$. The same holds replacing locally semi-concave with locally semi-convex.
\end{lm}
We now show how to construct $C^{1,1}$ critical subsolutions. Following the ideas of \cite{Be1}, we will apply successively the negative and positive Lax-Oleinik semi group, trying to perform this way a kind of Lasry-Lions regularization. Nevertheless, some difficulty arise. Let us begin with a lemma:

\begin{lm}
 Let $u$ be a continuous function which is strict outside of $ \widehat{\A}_u$, and $v$ verify that $$u\leqslant v\leqslant \ttt u+\alpha[0].$$ Assume moreover that  $u(x)=v(x)$ if and only if $x\in \A_u$ 
then $v$ itself is a critical subsolution, $v$ and $u$ coincide on $ \A_u= \A_v$ and $v$ is  strict outside of the set $ \widehat{\A}_u= \widehat{\A}_v$.
\end{lm}
\begin{proof}
 That $v$ is a subsolution is a direct consequence of the following inequality which comes from the monotony of the Lax-Oleinik semi-group
$$u\leqslant v\leqslant \ttt u+\alpha[0] \leqslant \ttt v+\alpha[0].$$
Now let us prove that $v$ is strict. Assume that for some $(x,y)\in M \times M$, the following holds: $v(x)-v(y)=c(y,x)+\alpha[0]$. Since $v$ is critically dominated we have $v(x)=\ttt v(x)+\alpha[0]$ and therefore, by the above inequality,
$$ v(x)= \ttt u(x)+\alpha[0] = \ttt v(x)+\alpha[0]$$
The following inequalities are also true:
\begin{eqnarray*}
c(y,x)+\alpha[0]&=& v(x)-v(y)\\
&=& \ttt u(x) +\alpha[0] -v(y)\\
&\leqslant&\ttt u(x)+\alpha[0]-u(y)\\
&\leqslant &u(y)+c(y,x)+\alpha[0]-u(y).
\end{eqnarray*}
Therefore all inequalities are equalities and $v(y)=u(y)$. By the assumption we made, this proves that $y\in \A_u$ and from $\ttt u(x)=u(y)+c(y,x)$ that $x\in \A_u$ too (by \ref{err}). Hence we have that $u(x)=\ttt u(x)+\alpha[0]$ which yields that $u(x)-u(y)=c(y,x)+\alpha[0]$ and finally that $(y,x)\in \widehat{\A}_u$ since $u$ is strict outside of $\widehat{\A}_u$.
Consequently, we have that $\A_v\subset \A_u$ and $v$ is strict outside of $\widehat{\A}_u$. Now, since $u=v$ on $\A_u$ (because $u=\ttt u+\alpha[0]$ on $\A_u$) we have in fact $\widehat{\A}_u=\widehat{\A}_v$ which finishes the proof.
\end{proof}
\begin{rem}\rm
In the previous lemma, a similar argument shows that the hypothesis $u(x)=v(x)$ if and only if $x\in \A_u$ may be replaced by the following one:
 $ \ttt u(x)+\alpha[0]=v(x)$ if and only if $x\in \A_u$.
\end{rem}

Therefore, given a critical subsolution $u$, by \ref{210},
we can construct a locally semi-concave critical subsolution $u_1$ which coincide with $u$ on $\A_u$ and which is strict outside of $\widehat{\A}_u$ and a locally semi-convex function $u_2$ having the same properties such that $u_2\leqslant u_1$ by setting $u_2=\TTT u_1-\alpha[0]$. Moreover, starting with $u$ strict outside of $ \widehat{\A}$ we are able to construct a locally semi-convex function $\TTT u-\alpha[0]$ and a locally semi-concave function $\ttt \TTT u$ which are both strict outside of $ \widehat{\A}$ and such that $\TTT u-\alpha[0]\leqslant \ttt \TTT u$. 
Now the idea will be to consider a $C^{1,1}$ function in between which is the one we are looking for.

\begin{Th}\label{strici}
If $u$ is a critical subsolution, then there exists a $C^{1,1}$ critical subsolution $u'$ such that $u$ and $u' $ coincide on $\A_u$ and $u'$ is strict outside of $\widehat{\A}_u$.\\
 There exists a $C^{1,1}$ critical subsolution which is strict outside of $\widehat{A}$.
\end{Th}
From the discussion above, the proof is a direct consequence of the following lemma which appears in \cite{Il}.  
\begin{Th}\label{inser}
 Given a locally semi-concave function $f:M\mapsto \R{}$ and a locally semi-convex function $g:M\mapsto \R{}$ such that $f\geqslant g$, there exists a $C^{1,1}$ function $h:M\mapsto \R{}$ such that $f\geqslant h\geqslant g$. Moreover, $h$ can be constructed in such a way that $h(x)=g(x)$ implies $f(x)=g(x)$.
\end{Th}

Let us mention that the previous theorem (\ref{inser}) is equivalent to Ilmanen's insertion lemma proved in \cite{carda}. Following Cardaliaguet's observation, two independent proofs of the claim were obtained in \cite{Beil,fz}.

We conclude this section by giving another analogue of Mather's graph theorem in this discrete setting. Let us define yet another Aubry set:
\begin{df}
Given a critical subsolution, let us set $\A_u^*\subset T^*M$ by 
$$\A_u^*=\Lambda_c^l (\widehat{\A}_u).$$
Finally, let us set 
$$\A^*=\Lambda_c^l (\widehat{\A}).$$
\end{df}
\begin{Th}[Mather's graph theorem bis]\label{graphbis}
Given a critical subsolution $u$, the canonical projection $\pi$ from $T^*M$ to $\R{}$ induces a bi-Lipschitz homeomorphism from $\A^*_u$ to $\A_u$.\\ 
The canonical projection $\pi$ from $T^*M$ to $\R{}$ induces a bi-Lipschitz homeomorphism from $\A^*$ to $\A$.
\end{Th}
\begin{proof}
By \ref{strici}, we can without loss of generality assume that $u$ is $C^{1,1}$. By \ref{newco} and by definition of the skew Legendre transform $\Lambda_c^l$, the application $\pi^{-1}$ from $\A_u$ to $\A_u^*$ is nothing but the following:
$$\forall x\in \A_u,\  \pi^{-1}(x)=(x,\d_x u)$$
which is therefore Lipschitz since $u$ is $C^{1,1}$.\\
The second part is proved similarly starting with a $C^{1,1}$ strict subsolution (given by \ref{strici})  whose Aubry set is $\A$.
\end{proof}

\section{Invariant and equivariant weak KAM solutions}\label{section5}

In this section, following the ideas of \cite{FaMa}, we consider the case of invariant cost
functions. This case arises naturally when studying covering spaces
with the group of deck transformations as group of symmetries (we will study this case in the next and last section). Let us notice that most results of this section can be proved in the much more general setting exposed in \cite{Za}, when $M$ is merely a length space at large scale. 

Let $G$
be a group of homeomorphisms that preserve $c$ that is
$$\forall g\in G, \forall (x,y)\in M\times M,\ c(g(x),g(y))=c(x,y).$$
We will denote by $\I$ the set of $G$-invariant functions that is 
$$\I=\left\{f\in \R{M},\forall g\in G,\ f\circ g=f\right\}.$$ 
For each
$C\in \R{}$ let 
$$\HH_{inv}(C)=\HH(C)\cap\I$$
 be the set of the
invariant functions which are $C$-dominated. It is clear that
$\HH_{inv}(C)\cap C^0(M,\R{})$ is a closed (for the topology of uniform convergence on compact subsets) and convex subset of
$\hh{C}$. It is also clear that, if $q$ denotes the canonical projection from $C^0(M,\R{})$ to $C^0(M,\R{})/\R{}\mathbbm{1}_M$ ($\mathbbm{1}_M$ denotes the constant function equal to $1$ on $M$), and if we let $\HHH(C)=q(\HH(C)\cap
C^0(X,\R{}))$, then we may define
 $$\HHH_{inv}(C)=q(\HH_{inv}(C)\cap
C^0(X,\R{}))=\HHH(C)\cap q(\I),$$
 where the last equality follows from the fact that $\I$ contains the constant
functions. Finally, since the Lax-Oleinik semi-group $\ttt$ commutes with the addition of constants, it induces canonically a semi-group  $\tttt$ on the quotient $C^0(M,\R{})/\R{}\mathbbm{1}_M$.
% and as seen previously, $H_{inv}(C)=H(C)\cap q(\I)$ is a
%convex, compact subet of $\HHH(c)$. We will also note $\II$ the
%quotient $q(\I)$.
\begin{pr}\label{invariance}
 If $u\in \I$, then $\ttt u\in \I$. Moreover, $\HH_{inv}(C)\neq
 \varnothing$ for all $C\geqslant C(0)$.
\end{pr}
\begin{proof}
 The last part of this proposition is immediate since constant
 functions are dominated by $c+C(0)\geqslant0$.\\ 
 To prove the first
 part, let $u\in \I$ and $g\in G$. Then
$$\ttt u(g(x))=\inf_{y\in M} u(y)+c(y,g(x))=\inf_{y\in M}
 u(g(y))+c(g(y),g(x))=\inf_{y\in M} u(y)+c(y,x)$$
  where we have first
 used the fact that $g$ is a bijection and then the invariance of $u$
 and $c$ by $g$.
\end{proof}
We now define the invariant critical value for the action of the group
$G$ as the constant
$$C_{inv}=\inf\{C\in \R{},\ \HH_{inv}(C)\neq \varnothing\}.$$
Clearly, we have that $-A(0)\leqslant \alpha [0]\leqslant C_{inv}\leqslant C(0)$.
We are now able to prove the invariant weak KAM theorem:
\begin{Th}[invariant weak KAM]\label{invkam}
 There exists a $G$-invariant function $u$ such that $u=\ttt u+C_{inv}$.
\end{Th}

\begin{proof}
We only sketch the proof since it is very similar to the proof of the weak KAM theorem (\cite{Za}). We know that $\I$ is stable by $\ttt$. This implies that $\II$ is stable by $\tttt$. Therefore $\HHH_{inv}(C)$ is stable by $\tttt$ and so is $H_{inv}(C)=\overline{\conv (\tttt(\HHH_{inv}(C)))}$, for each $C\in \R{}$. It is obvious that $H_{inv}(C)\neq \varnothing$ if and only if $\HHH_{inv}(C)\neq\varnothing$. It can be checked, using the Ascoli theorem, that  $H_{inv}(C)$ is convex and compact for the quotient of the topology of uniform convergence on compact subsets.
As a consequence, 
$$\bigcap_{C>C_{inv}}H_{inv}(C)\neq\varnothing$$
 as the intersection of a decreasing family of compact nonempty sets. Therefore, $\HHH_{inv}(C_{inv})$ is nonempty. Moreover, $\tttt$ induces a continuous mapping from $H_{inv}(C_{inv})$ into itself, so applying the Schauder-Tykhonoff theorem, we obtain a fixed point, that is a function $u_{inv}\in \HH_{inv}(C_{inv})$ and a constant $C'$ such that $\ttt u_{inv}=u_{inv}+C'$. Finally, using the minimality of $C_{inv}$, it is easy to prove that in fact $-C'=C_{inv}$ which ends the proof of the theorem. 
\end{proof}

Instead of looking at functions invariant by the group of symmetries
$G$ we can consider functions whose projections to
$C^0(X,\R{})\backslash \R{}\mathbbm{1}_M$ are invariant that is functions
$u$ such that for each $g\in G$ there is a $\rho (g)$ such that
$u\circ g =u+\rho(g)$. Obviously, $\rho :G\rightarrow \R{}$ is a group
homomorphism. We will denote by $\Hom(G,\R{})$ the set of group
homomorphisms from $G$ to $\R{}$. Given a $\rho \in \Hom(G,\R{})$ we
will say that a function $u$ is $\rho$-equivariant if it satisfies
$u\circ g =u+\rho(g)$ for all $g$ in $G$, we will denote by
$\I_{\rho}$ the set of continuous $\rho$-equivariant functions. It is
obvious that $\I_{\rho}$ is an affine subset of $C^0(X,\R{})$, in
fact, it is either empty or equal to $u+\I$ where $u\in \I_{\rho}$. In
particular $\I_0=\I$.  For $C\in \R{}$, $\rho \in \Hom(G,\R{})$, we
set $\HH_{\rho}(C)=\hh{C}\cap \I_{\rho}$ and we define the $\rho$-equivariant critical value 
$$C_\rho=\inf\{C\in\R{},\ \HH_{\rho}(C)\neq\varnothing\}\in
\R{}\cup\{+\infty\}.$$
 Notice that the value $+\infty$ is reached if
and only if there is no $C$ such that
$\HH_{\rho}(C)\neq\varnothing$. For example, the $0$-equivariant critical value or invariant critical value is nothing but $C_0=C_{inv}$.
% The function $\alpha :\rho \mapsto
%\alpha [\rho]$ is called the Mather function.
First, we notice that since the Lax-Oleinik semi-group commutes with addition of constants, we have, as in \ref{invariance}, the following:
\begin{pr}
Let us consider a morphism $\rho\in \Hom(G,\R{})$.   If $u\in \I_{\rho}$, then $\ttt u\in \I_{\rho}$. 
\end{pr}
  
\begin{df}\rm
We will say that a
homomorphism $\rho :G\rightarrow \R{}$ is tame if the inequality $C_\rho<+\infty$ is verified and we will denote by $\Homt(G,\R{})$ the set of tame
homomorphisms.
\end{df}
  Since $\I_{\rho}$ is closed for the compact open
topology and invariant by the Lax-Oleinik semi-group, we can easily
adapt the proof of \ref{invkam} to obtain the following equivariant
weak KAM theorem:
\begin{Th}[equivariant weak KAM]\label{equikam}
 For each $\rho \in \Homt(G,\R{})$, we have $\HH_{\rho}(C_\rho)\neq\varnothing$. Moreover, we can find a $\rho$-equivariant
 weak KAM solution in $\HH_{\rho}(C_\rho)$ that is a continuous
 function $u$ such that $u=\ttt u+C_\rho$ and for all $g\in G$,
 $u\circ g=u+\rho(g)$.
\end{Th}
Here are some properties of tame homomorphisms and of the function $\rho \mapsto C_\rho$.
\begin{pr}
 The set $\Homt(G,\R{})$ is a vector subspace of $\Hom(G,\R{})$. The
 restriction of the function $C$ to $\Homt(G,\R{})$ is
 convex. Moreover, if $\Homt(G,\R{})$ is finite dimensional, then
 the  function $C $ is super-linear.
\end{pr}
\begin{proof}
 Let $\rho _1$ and $\rho_2$ be two tame homomorphisms, $\lambda_1$ and
 $\lambda_2$ be real numbers. Let $u_1\in \HH_{\rho_1}(C_1)$ and
 $u_2\in \HH_{\rho_2}(C_2)$ where $C_1$ and $C_2$ have been chosen
 such that $\HH_{\rho_1}(C_1)\neq \varnothing$ and
 $\HH_{\rho_2}(C_2)\neq \varnothing$. Then $\lambda_1 u_1+\lambda_2
 u_2\in \I_{\lambda_1\rho_1+\lambda_2\rho_2}$ (as a matter of fact,
 $\lambda_1\I_{\rho_1}+\lambda_2\I_{\rho_2}\subset
 \I_{\lambda_1\rho_1+\lambda_2\rho_2}$).  Moreover, we clearly have
 that $\lambda_1 u_1+\lambda_2 u_2\in
 \HH(|\lambda_1|C_1+|\lambda_2|C_2)$ which proves that $\Homt(G,\R{})$
 is a vector subspace of $\Hom(G,\R{})$.
 
  If now
 $\lambda_1,\lambda_2\geqslant 0$ and $\lambda_1+\lambda_2=1$ then the inclusion
 $$\lambda_1\HH(C_1)+\lambda_2\HH(C_2)\subset
 \HH(\lambda_1C_1+\lambda_2 C_2)$$
 holds. Altogether with the inclusion
 $$\lambda_1\I_{\rho_1}+\lambda_2\I_{\rho_2}\subset
 \I_{\lambda_1\rho_1+\lambda_2\rho_2},$$
 this proves the convexity of the
  function $C$.
 
 We now prove the super-linearity when
 $\Homt(G,\R{})$ is finite dimensional. For each $g\in G$, consider
 the linear form 
 $$\hat{g}:\Homt(G,\R{})\rightarrow \R{},$$
 $$ \rho \mapsto
 \rho(g).$$
  These linear forms span a sub-vector space of the dual of
 $\Homt(G,\R{})$ which is therefore finite dimensional. Let
 $g_1,\ldots,g_k$ be such that any $\hat{g}$ is a linear combination
 of the $\hat{g_i}$. In particular, it follows that if $\rho \in
 \Homt(G,\R{})$ then $\rho=0$ if only if $\rho(g_1)=\cdots
 =\rho(g_k)=0$. Thus we can use as a norm on $\Homt(G,\R{})$, $\|\rho
 \|=\max_{i=1}^k|\rho(g_i)|$. If $\rho$ is given, let $u$ be a
 $\rho$-equivariant weak KAM solution such that $u=\ttt u +C_\rho$. We have $n\rho(g_i)=\rho(g_i^n)=u(g_i^n(x_0))-u(x_0)$
        for $n\in \mathbb{N}$, $i=1,\ldots ,k$ and some $x_0$
        fixed. We now have using the domination $u\<c+C_\rho$
$$n\rho (g_i)=u(g_i^n(x_0))-u(x_0)\leqslant c(x_0,g_i^n(x_0))+C_\rho.$$
 The constant $A_{i,n}=c(x_0,g_i^n(x_0))$ is independant of
$\rho$. Arguing in the same way as above with $g_i^{-1}$ instead of
$g_i$, we obtain a constant $A'_{i,n}$ independant of $\rho$ such that
$$-n\rho (g_i)=u(g_i^{-n}(x_0))-u(x_0)\leqslant A'_{i,n}+C_\rho.$$
 If we set $A_n=\max(A_{1,n},\ldots,A_{k,n},A'_{1,n},\ldots
,A'_{k,n})$ we have obtained a constant independant of $\rho$ such
that
$$n\|\rho\|=n\max(\rho(g_1),\ldots,\rho(g_k),-\rho(g_1),\ldots,-\rho(g_k))\leqslant
A_n+C_\rho.$$ 
Since $n$ is an arbitrary integer, this proves the
super-linearity of $\rho \mapsto C_\rho$.

\end{proof}
We set 
$$C_{G,min}=\inf\{C_\rho,\ \rho\in \Hom(G,\R{})\}=\inf\{C_\rho,\ \rho\in \Homt(G,\R{})\}.$$
\begin{lm}
 There exists $\rho \in \Homt(G,\R{})$ such that $C_{G,min}=C_\rho$.
\end{lm}
\begin{proof}
 Of course, when $\Homt(G,\R{})$ is finite dimensional, this follows
 from the super-linearity of the  function $C$.
 
  For the general
 case, pick a decreasing sequence $C_{\rho_n}$ which converges to
 $C_{G,min}$. For each $n\in \mathbb{N}$, pick a function $u_n\in
 \ttt(\HH_{\rho_n}(C_{\rho_n}))$. The functions are locally equicontinous
 because they all belong to $\ttt(\HH(C_{\rho_0}))$. Substracting
 a constant from each $u_n$ and extracting a subsequence if necessary,
 we can assume that $u_n$ converges uniformly on each compact subset
 of $M$ to a function $u$. Since for $n\geqslant n_0$, $u_n$ is in the
 closed set $\HH(C_{\rho_{n_0}})$, we must have $u\in \HH(C_{\rho_{n_0}})$ for each $n_0$. Hence, $u\in \HH(C_{G,min})$. Since
 for $x\in M$ we have $\rho_n (g)=u_n(g(x))-u_n(x)$ we conclude that
 $\rho_n$ converges (pointwise) to a $\rho \in \Hom(G,\R{})$ and $u\in
 \I_{\rho}$. It follows that $C_\rho\leqslant C_{G,min}$ but
 the reverse inequality follows from the definition of $C_{G,min}$.
\end{proof}

\section{Application: Mather's $\alpha$ function on the cohomology}\label{section6}

In this final section, following Mather's ideas (\cite{Ma2}), we apply the preceding results to the case when the group of symmetries rise from a covering of $M$. Let us consider $M$ a smooth, finite dimensional, connected riemmanian manifold, $g_M$ its metric. Let $\MM$ be its  covering space verifying 
$$\pi_1 \left(\MM \right)=\ker(\mathfrak{H})$$
 where $\mathfrak{H}:\pi_1(M)\to H_1(M,\R{})$ is the Hurewicz homomorphism. We consider then a cost function $\cc : \MM \times \MM \to \R{}$ which verifies \ref{unif} and \ref{unifb}. Let us assume moreover that $\cc$ is invariant by the diagonal action of the group of deck transformations $\mathfrak{T}$. This means that if $T$ is a deck transformation, the following holds:
 $$\forall (\xx,\yy)\in M\times M, \  \cc(\xx,\yy)=\cc(T(\xx),T(\yy)).$$
 Let $p: \MM \to M$ be the cover, we may define a cost function $c:M\times M \to \R{}$ by
 $$\forall (x,y)\in M\times M, \ c(x,y)=\inf_{\substack{p(\xx)=x \\p(\yy)=y}}\cc(\xx,\yy).$$
 \begin{pr}\label{cocontinuity}
 The cost function $c$ is continuous, uniformly super-linear and uniformly bounded in the sense of \ref{unif} and \ref{unifb}. Moreover, if $(x,y)\in M\times M$ then for each $\xx\in \MM$ verifying $p(\xx)=x$ there is a $\yy\in \MM$ such that $p(\yy)=y$ and $c(x,y)=\cc(\xx,\yy)$.
 \end{pr}
 
 \begin{proof}
 The proof of the continuity of $c$ is much similar to the proofs of  regularity of the Lax-Oleinik semi-groups (see \cite{Za}) therefore we will sketch it briefly. Let us consider $K\subset M$ a compact subset of $M$ and $\KK \subset \MM$  compact verifying $p(\KK)=K$. Since $\cc$ is invariant by the diagonal action of the group of deck transformations $\mathfrak{T}$ we have the following:
 $$\forall (x,y)\in K\times M, \ c(x,y)=\inf_{\substack{\xx\in \KK, p(\xx)=x \\p(\yy)=y}}\cc(\xx,\yy).$$
 Let us now consider another compact set $K_1\subset M$ . It may be proved, using the super-linearity of $\cc$, that there exists a compact set $\KK_1$ such that $K_1\subset p(\KK_1)$ and 
 $$\forall (x,y)\in K\times K_1, \ c(x,y)=\inf_{\substack{\xx\in \KK, p(\xx)=x \\ \yy\in \KK_1,p(\yy)=y}}\cc(\xx,\yy).$$
 Since $\KK\times \KK_1$ is  compact, the function $\cc$ restricted to $\KK\times \KK_1$ is uniformly continuous and the function $c$ restricted to $K\times K_1$ is a finite infimum (in fact this infimum is achieved) of uniformly continuous functions, therefore it is continuous. Note that since we managed to restrict ourselves to compact sets, we may apply the previous result to $K=\{x\}$ and $\KK =\{\xx\}$ to obtain the last point of the proposition. 
 
 Let $\d(.,.)$ be the riemannian distance on $M$ and $\dd(.,.)$ the induced distance on $\MM$. The following is verified:
 $$\forall (\xx,\yy)\in \MM\times \MM, \ \d(p(\xx),p(\yy))\leqslant \dd(\xx,\yy).$$
 Since $\cc$ is uniformly super-linear we have that for every $k\geqslant 0$,
   there exists $C(k)\in \R{}$ such that 
   $$\forall (\xx,\yy)\in \MM \times \MM,\ 
   \cc(\xx,\yy)\geqslant k\dd(\xx,\yy)-C(k).$$
   Let us pick $(x_0,y_0)\in M\times M$ and $(\xx_0,\yy_0)$ such that $p(\xx_0)=x_0$, $p(\yy_0)=y_0$ and $c(x_0,y_0)=\cc(\xx_0,\yy_0)$. The following holds:
   $$c(x_0,y_0)=\cc(\xx_0,\yy_0)\geqslant k\dd(\xx_0,\yy_0)-C(k)\geqslant k\d (x_0,y_0) -C(k),$$
   which proves the super-linearity of $c$.
   
   Similarly, for every $R\in \R{}$, there
  exists $A(R)\in \R{}$ such that 
  $$\dd(\xx,\yy)\leqslant R \Rightarrow
  \cc(\xx,\yy)\leqslant A(R).$$
   If $\d(x_0,y_0)\leqslant R$, we can find $(\xx_0,\yy_0)$ such that $p(\xx_0)=x_0$, $p(\yy_0)=y_0$ and $\d(x_0,y_0)=\dd(\xx_0,\yy_0)\leqslant R$. Therefore, using the definition of $c$ we obtain 
$$c(x_0,y_0)\leqslant \cc(\xx_0,\yy_0)\leqslant A(R)$$ which proves that $c$ is uniformly bounded in the sense of \ref{unifb}.
 \end{proof}
 
 Let us now consider a bounded (with respect to the metric $g_M$) closed $1$-form $\om$ on $M$. This form lifts to an exact form $\tilde{\om}=\d\tilde{f}$ on $\MM$. Moreover, the function $\tilde{f}$ is globally Lipschitz hence has linear growth. We may therefore define a cost function $\cc_{\tilde{\om}}$ by
 $$ \forall (\xx,\yy)\in \MM \times \MM,\ 
   \cc_{\tilde{\om}}(\xx,\yy)=\cc(\xx,\yy)-\tilde{f}(\yy)+\tilde{f}(\xx).$$
   Note that this cost function is still super-linear and uniformly bounded and that it does not depend on the choice of the primitive $\tilde{f}$. Let us fix a point $\xx\in \MM$ and define now the morphism $\rho_{\tilde{\om}}:\mathfrak{T}\to \R{}$ by
   $$\forall T\in \mathfrak{T},\  \rho_{\tilde{\om}}(T)=\tilde{f}(T(x))-\tilde{f}(x).$$
   It is straightforward to check that $ \rho_{\tilde{\om}}$ is indeed a morphism and that it is independent  from $x$ by Stoke's formula. Finally, the map $\om \to \rho_{\tilde{\om}}$ is linear in $\om$ and vanishes if and only if  $\om$ is exact. Therefore it induces an injective morphism from the $g_M$-bounded cohomology of order $1$, $H^1_{g_M,b}(M,\R{})$, to $\Hom (\mathfrak{T},\R{})$. We still denote by $\rho$ this morphism. We now have the following lemma:
   \begin{lm}\label{inclusion}
  The following inclusion holds:
   $$\mathrm{Im}(\rho)\subset \Homt(\mathfrak{T},\R{}).$$ 
   \end{lm}

\begin{proof}
It follows from the discussion above that, if $[\om]\in H^1_{g_M,b}(M,\R{})$ and $\om$ is a bounded $1$-form whose cohomology class is $[\om]$ then $\cc_{\tilde{\om}}$ verifies \ref{unif} and \ref{unifb}. Therefore, by the invariant weak KAM theorem (\ref{invkam}) applied to the cost $\cc_{\tilde{\om}}$  there exist a function $\tilde{u}$ and a constant $C$ such that $\tilde{u}=\tTT \tilde{u}+C$ and $\tilde{u}\in \I$. This means exactly that $\tilde{u}+\tilde{f}=\tT (\tilde{u}+\tilde{f})+C$ and $\tilde{u}+\tilde{f}\in \I_{\rho_{\tilde{\om}}}$.
\end{proof}

We now introduce Mather's alpha function:
\begin{df}\rm
Let $[\om]\in H^1_{g_M,b}(M,\R{})$ be the cohomology class of a a bounded $1$-form $\om$, we define the constant $\alpha[\om]\in \mathbb{R}$ by the relation $\alpha[\om]=C_{\rho_{\tilde{\om}}}$. In other words, the value $\alpha[\om]$ is the invariant critical value of the cost $\cc_{\tilde{\om}}$.
\end{df}

In an analogous way to what we already did, if $\om$ is a closed bounded $1$-form on $M$, we may define a cost function $c_\om$ by
 $$\forall (x,y)\in M\times M, \ c_\om(x,y)=\inf_{\substack{p(\xx)=x \\p(\yy)=y}}\cc_{\tilde{\om}}(\xx,\yy).$$
 The constant $\alpha[\om]$ is also the critical value of the cost $c_\om$. Moreover, this constant depends only on the cohomology class $[\om]$ of the form $\om$. As a matter of fact, as in the proof of \ref{inclusion}, if $\om=\d f$ is exact, then $u:M\to \R{}$ is a critical subsolution for $c_\om$ if and only if $u+f$ is a critical subsolution for $c$. This also justifies a posteriori the notation $\alpha[\om]$.
 
 From now on, we will assume, without loss of generality, that all the forms considered are smooth. The end of this paper will be devoted to checking that it is possible to adapt the machinery of sections \ref{known}, \ref{2} and \ref{existence} to this cohomological setting.

\begin{pr}\label{verification}
Assume the cost $\cc :\MM \times \MM \to \R{}$ is locally semi-concave then the cost $c:M\times M\to \R{}$ is also locally semi-concave. Assume moreover that $\cc$ verifies the left and right twist conditions, then so does $c$. Finally, in the latter case, if $\om$ is a smooth  closed $1$-form on $M$, the costs $\cc_{\tilde{\om}}$ and $c_{\om}$ are locally semi-concave and verify the left and right twist conditions.
\end{pr}

\begin{proof}
As in the proof of \ref{cocontinuity}, the function $c$ is locally semi-concave because it is locally a finite infimum of equi-semi-concave functions (everything can locally be reduced to taking infimums over relatively compact sets).

For the second part of the proposition, let us prove only the left twist condition. Consider  a point $x_0\in M$ and a lift $\xx_0\in \MM$ such that $p(\xx_0)=x_0$. By the last part of \ref{cocontinuity}, the following holds:
$$\forall y\in M,\ c(x_0,y)=\inf_{\yy\in p^{-1}\{y\}} \cc(\xx_0,\yy).$$
Assume now that for some $y\in M$ the partial derivative $\partial c/\partial x (x_0,y)$ exists and consider $\yy\in \MM$ such that $c(x_0,y)=\cc(\xx_0,\yy)$. Since $\cc$ is locally semi-concave, it follows that the partial derivative $\partial \cc/\partial \xx (\xx_0,\yy)$ also exists and verifies (identifying the cotangent fibers $T_{(\xx_0,\yy)}\MM\times \MM$ and $T_{(x_0,y)}M\times M$ via the cover $p$ which is a local diffeomorphism)
\begin{equation}\label{twist}
\frac{\partial \cc}{\partial \xx} (\xx_0,\yy)=\frac{\partial c}{\partial x} (x_0,y).
\end{equation}
Now, since $\cc$ verifies the left twist condition, it follows that the map 
$$\yy\mapsto\Lambda_{\cc}^l(\xx_0,\yy)=\left(\xx_0,-\frac{\partial \cc}{\partial \xx} (\xx_0,\yy)\right)$$
 is injective on its domain of definition, and it follows immediately from \ref{twist} that the left Legendre transform
 $$y\mapsto\Lambda_c^l(x_0,y)=\left(x_0,-\frac{\partial c}{\partial x} (x_0,y)\right)$$
 is also injective on its domain of definition, which means that $c$ verifies the left twist condition. 
 
 The last part of the proposition is now straightforward. Indeed, if $\om$ is smooth, then so will be the function $\tilde{f}$, and the function 
 $$\cc_{\tilde{\om}}:(\xx,\yy)\mapsto \cc(\xx,\yy)-\tilde{f}(\yy)+\tilde{f}(\xx)$$
  remains locally-semi-concave. Moreover the left Legendre transform of $\cc_{\tilde{\om}}$ is defined if and only if the left Legendre transform of $\cc$ is defined and it is given by the formula 
 $$\Lambda_{\cc_{\tilde{\om}}}^l(\xx,\yy)=\left(\xx,-\frac{\partial \cc_{\tilde{\om}}}{\partial x} (\xx,\yy)\right)=\left(\xx,-\frac{\partial \cc}{\partial x} (\xx,\yy)-\d_{\xx} \tilde{f}\right)$$
 which clearly gives that $\cc$ verifies the left twist condition if and only if $\cc_{\tilde{\om}}$ does.
\end{proof}

Thanks to \ref{verification}, it is possible to associate to each cohomology class $[\om]\in H^1_{g_M,b}(M,\R{})$ Aubry sets $\A_{[\om]}$, $\widehat{\A}_{[\om]}$ and $\AA_{[\om]}$ by using the already introduced notions to the cost $c_{\om}$. Notice that these sets depend only on the cohomology class for, as in the time-continuous case (see \cite{Ma2}), minimizers with fixed endpoints are unchanged by the addition of an exact form to the cost $c$. Theorem \ref{strici} then applies, proving the existence of $C^{1,1}$ strict subsolutions associated to each cohomology class.

\bibliography{CC-09-08}
\bibliographystyle{alpha}

\end{document}